\documentclass[final,3p]{elsarticle}
\usepackage{geometry}
\usepackage{mathrsfs,enumerate,amssymb,amsmath,bm,amsthm,color,lineno}
\usepackage[all]{xy}
\usepackage{latexsym}
\usepackage{amscd,verbatim}
\usepackage{graphicx}
\usepackage[colorlinks=true]{hyperref}
\usepackage{amssymb}
\usepackage{placeins}
\usepackage{float}
\usepackage{varioref}

\theoremstyle{plain}
\newtheorem{theorem}{Theorem}[section]
\newtheorem{lemma}{Lemma}[section]
\newtheorem{proposition}{Proposition}[section]
\newtheorem{corollary}{Corollary}[section]
\theoremstyle{definition}
\newtheorem{definition}{Definition}[section]

\newtheorem{question}{Question}

\numberwithin{equation}{section}
\newtheorem{remark}{Remark}
\journal{}
\usepackage{pifont}
\newcommand{\hollowstar}{\text{\ding{78}}}
\newcommand{\fivestar}{\text{\ding{79}}}

\begin{document}

\begin{frontmatter}
\title{On the existence of $t_r$-norm and $t_r$-conorm not in
convolution form\tnoteref{mytitlenote}}
\tnotetext[mytitlenote]{This work was supported by National Natural Science Foundation of China
(No. 11601449), Yibin University Cultivating Foundation (No. 02180199),
Science and Technology Innovation Team of Education Department of Sichuan for
Dynamical System and its Applications (No. 18TD0013), Youth Science and Technology Innovation
Team of Southwest Petroleum University for Nonlinear Systems (No. 2017CXTD02), and
Key Natural Science Foundation of Universities in Guangdong Province (No. 2019KZDXM027).}

\author[a1,a2,a3]{Xinxing Wu\corref{mycorrespondingauthor}}
\cortext[mycorrespondingauthor]{Corresponding author}
\address[a1]{School of Mathematics and Statistics, Guizhou University of Finance and Economics, Guiyang 550025, China}
\address[a2]{Institute for Artificial Intelligence, Southwest Petroleum University, Chengdu, Sichuan 610500, China}
\address[a3]{Zhuhai College of Jilin University, Zhuhai, Guangdong 519041, China}
\ead{wuxinxing5201314@163.com}

\author[a4]{Guanrong Chen}
\address[a4]{Department of Electrical Engineering, City University of Hong Kong,
Hong Kong SAR, China}
\ead{gchen@ee.cityu.edu.hk}

\author[a3]{Lidong Wang}
\ead{wld0707@126.com}

\begin{abstract}
This paper constructs a $t_{r}$-norm and a $t_{r}$-conorm on the set of all normal and
convex functions from ${[0, 1]}$ to ${[0, 1]}$, which are not obtained by using the following
two formulas on binary operations
${\curlywedge}$ and ${\curlyvee}$:
$$
{(f\curlywedge g)(x)=\sup\left\{f(y)\ast g(z)\mid y\vartriangle z=x\right\},}
$$
$$
{(f\curlyvee g)(x)=\sup\left\{f(y)\ast g(z)\mid y\ \triangledown\ z=x\right\},}
$$
where ${f, g\in Map([0, 1], [0, 1])}$, ${\vartriangle}$ and ${\triangledown}$ are respectively
a ${t}$-norm and a ${t}$-conorm on ${[0, 1]}$, and ${\ast}$ is a binary operation on ${[0, 1]}$.
{\color{blue}This result answers affirmatively an open problem posed in \cite{HCT2015}. Moreover, the duality
between $t_r$-norms and  $t_r$-conorms is obtained by the introduction of operations dual to binary
operations on ${Map([0, 1], [0, 1])}$.}
\end{abstract}


\end{frontmatter}

\section{Introduction}
In 1975, Zadeh~\cite{Z1975} introduced the notion of type-2 fuzzy sets (T2FSs), that is,
a fuzzy set with fuzzy sets as truth values (simply, ``fuzzy-fuzzy sets''), as an extension
of type-1 fuzzy sets (FSs) and interval-valued fuzzy sets (IVFSs), which was then equivalently
expressed in different forms by Mendel et al. \cite{KM2001,KM2001a,M2001,MJ2002}.
{\color{blue}The definitions of triangular norms (briefly $t$-norms) and triangular conorms
(briefly $t$-conorms) on the real unit interval were introduced by Schweizer
and Sklar \cite{SS1961} in the framework of probabilistic metric spaces.
These definitions exploits the main idea of Menger \cite{Me1942} that extends
the classical triangle inequality in metric spaces to probabilistic metric
spaces.} In 2006, Walker and Walker \cite{WW2006} extended $t$-norm and $t$-conorm on $I$ to the algebra
of truth values on T2FSs and IVFSs. Then, Hern\'{a}ndez et al.~\cite{HCT2015} modified Walker
and Walker's definition and introduced the notions of {\color{blue}a $t_{r}$-norm and a $t_{r}$-conorm}
by adding some ``restrictive axioms'' (see Definition~\ref{Def-1} below). In particular,
{\color{blue}in \cite{HCT2015} they proved} that the binary operation $\curlywedge$ (resp., $\curlyvee$)
on the set $\mathbf{L}$ of all normal and convex functions is a $t_{r}$-norm (resp., a $t_{r}$-conorm).
They also proposed the following two open problems on the binary operations $\curlywedge$ and
$\curlyvee$ (see Definition~\ref{HCT-Def} below).

\begin{question}\label{Q-2}\cite{HCT2015}
{\it Apart from the $t$-norm, does there exist other binary operation
`$\ast$' on $I$ such that `$\curlywedge$' and `$\curlyvee$' are, respectively,
a $t_r$-norm and a $t_r$-conorm on $\mathbf{L}$?}
\end{question}

\begin{question}\label{Q-1}\cite{HCT2015}
{\it Determine other binary operations, which
are not obtained using the formulas given for the operations `$\curlywedge$'
and `$\curlyvee$', that are either a $t_r$-norm or a $t_r$-conorm on $\mathbf{L}$.}
\end{question}

Recently, {\color{blue}in \cite{WC-TFS} we have} answered negatively Question~\ref{Q-2},
proving that, if a binary operation $\ast$ ensures $\curlywedge$ be a $t_{r}$-norm on $\mathbf{L}$
or $\curlyvee$ be a $t_{r}$-conorm on $\mathbf{L}$, then $\ast$ is a $t$-norm. This paper is
devoted to solving Question~\ref{Q-1} by constructing a $t_r$-norm `$\hollowstar$'
(see Section~\ref{S-V}) and a $t_r$-conorm `$\fivestar$' (see Section~\ref{S-VI}) on $\mathbf{L}$,
which cannot be obtained by the formulas defining the operations `$\curlywedge$'
and `$\curlyvee$'.

\section{Preliminaries and basic concepts}
Throughout this paper, let $I=[0, 1]$, $Map(X, Y)$ be the set of all mappings from $X$ to $Y$,
and `$\leq$' denote the usual order relation in the lattice of real numbers, with $\mathbf{M}=Map(I, I)$.
Let $\vee$ and $\wedge$ be the maximum and minimum operations, respectively, on lattice $I$.

\begin{definition}\cite{Z1965}
A {\it type-1 fuzzy set} $A$ in space $X$ is a mapping from $X$ to $I$, i.e.,
$A\in Map(X, I)$.
\end{definition}

\begin{definition}\cite{WW2005}
A {\it type-2 fuzzy set} $A$ in space $X$ is a mapping
$$
A: X\rightarrow \mathbf{M},
$$
i.e., $A\in Map(X, \mathbf{M})$.
\end{definition}

{\color{blue}\begin{definition}\cite{WW2005}
A function $f\in \mathbf{M}$ is
\begin{itemize}
  \item[(1)] {\it normal} if $\sup\{f(x)\mid x\in I\}=1$;
  \item[(2)] {\it convex} if, for any $0\leq x\leq y\leq z\leq 1$, $f(y)\geq f(x)\wedge f(z)$.
\end{itemize}
\end{definition}
}

Let $\mathbf{N}$ and $\mathbf{L}$ denote the set of all normal functions in $\mathbf{M}$ and
the set of all normal and convex functions in $\mathbf{M}$, respectively.

For any subset $B$ of $X$, a special fuzzy set $\bm{1}_{B}$, called the {\it characteristic function}
of $B$, is defined by
$$
\bm{1}_{B}(x)=\begin{cases}
1, & x\in B,\\
0, & x\in X\backslash B.
\end{cases}
$$
Let
$\mathbf{J}=\{\bm{1}_{\{x\}}\mid x\in I\}$ and $\mathbf{K}=\{\bm{1}_{[a, b]}\mid 0\leq a \leq b\leq 1\}$.

\begin{definition}\cite{KMP2000}
A binary operation $\ast: I^{2}\rightarrow I$ is a {\it $t$-norm} on $I$ if it satisfies the following axioms:
\begin{enumerate}
  \item[(T1)] ({\it commutativity}) $x\ast y=y\ast x$ for $x, y\in I$;
  \item[(T2)] ({\it associativity}) $(x \ast y) \ast z=x\ast (y\ast z)$ for $x, y, z\in I$;
  \item[(T3)] ({\it monotonicity}) $\ast$ is increasing in each argument;
  \item[(T4)] ({\it neutral element}) $1\ast x=x\ast 1=x$ for $x\in I$.
\end{enumerate}
A binary operation $\ast: I^2\rightarrow I$ is a
{\it $t$-conorm} on $I$ if it satisfies axioms (T1), (T2), and (T3) above;
and axiom (T4'): $0\ast x=x\ast 0=x$ for $x\in I$.
\end{definition}


\begin{definition}\cite{HCT2015}\label{HCT-Def}
Let $\ast$ be a binary operation on $I$, $\vartriangle$ be a
$t$-norm on $I$, and $\triangledown$ be a $t$-conorm on $I$. Define the binary operations
$\curlywedge$ and $\curlyvee: \mathbf{M}^2\rightarrow \mathbf{M}$ as follows: for $f, g\in \mathbf{M}$,
\begin{equation}\label{O-1}
(f\curlywedge g)(x)=\sup\left\{f(y)\ast g(z)\mid y\vartriangle z =x\right\},
\end{equation}
and
\begin{equation}\label{O-2}
(f\curlyvee g)(x)=\sup\left\{f(y)\ast g(z)\mid y\ \triangledown\ z =x\right\}.
\end{equation}
\end{definition}

\begin{definition}\cite{WW2005}\label{Def-7}
The operations of $\sqcup$ ({\it union}), $\sqcap$ ({\it intersection}),
$\neg$ ({\it complementation}) on $\mathbf{M}$ are
defined as follows: for $f, g\in \mathbf{M}$,
\begin{align*}
(f\sqcup g)(x)&=\sup\{f(y)\wedge g(z)\mid y\vee z=x\},\\
(f\sqcap g)(x)&=\sup\{f(y)\wedge g(z)\mid y\wedge z=x\},
\end{align*}
and
$$
(\neg f)(x)=\sup\{f(y)\mid 1-y=x\}=f(1-x).
$$
\end{definition}

From \cite{WW2005}, it follows that
$\mathfrak{M}=(\mathbf{M}, \sqcup, \sqcap, \neg, \bm{1}_{\{0\}}, \bm{1}_{\{1\}})$
does not have a lattice structure, although $\sqcup$ and $\sqcap$ satisfy the
De Morgan's laws with respect
to the complementation $\neg$.

Walker and Walker \cite{WW2005} introduced the following partial orders $\sqsubseteq$
and $\preccurlyeq$ on $\mathbf{M}$.
\begin{definition}\cite{WW2005}
$f\sqsubseteq g$ if $f\sqcap g=f$; $f\preccurlyeq g$ if $f\sqcup g=g$.
\end{definition}

It follows from \cite[Proposition 14]{WW2005} that both $\sqsubseteq$ and $\preccurlyeq$ are partial
orders on $\mathbf{M}$. Generally, the partial orders $\sqsubseteq$ and $\preccurlyeq$ do not coincide.
{\color{blue}In \cite{HWW2010,HWW2008,WW2005}}, it was proved that $\sqsubseteq$ and $\preccurlyeq$ coincide on $\mathbf{L}$,
and the subalgebra $\mathfrak{L}=(\mathbf{L}, \sqcup, \sqcap, \neg, \bm{1}_{\{0\}}, \bm{1}_{\{1\}})$
is a bounded complete lattice. In particular, $\bm{1}_{\{0\}}$ and $\bm{1}_{\{1\}}$
are {\color{blue}the minimum and the maximum} of $\mathfrak{L}$, respectively.

\begin{definition}\cite{HCT2015}\label{Def-1}
A binary operation $\widetilde{T}: \mathbf{L}^2 \rightarrow \mathbf{L}$ is a {\it $t_{r}$-norm}
($t$-norm according to the restrictive axioms), if
\begin{itemize}
  \item[(O1)] $\widetilde{T}$ is commutative, i.e., $\widetilde{T}(f, g)=\widetilde{T}(g, f)$ for $f, g\in \mathbf{L}$;
  \item[(O2)] $\widetilde{T}$ is associative, i.e., $\widetilde{T}(\widetilde{T}(f, g), h)=\widetilde{T}(f, \widetilde{T}(g, h))$
  for $f, g, h\in \mathbf{L}$;
  \item[(O3)] $\widetilde{T}(f, \bm{1}_{\{1\}})=f$ for $f\in \mathbf{L}$ (neutral element);
  \item[(O4)] $\widetilde{T}$ is increasing, i.e., for $f, g, h\in \mathbf{L}$ such that $f\sqsubseteq g$,
  $\widetilde{T}(f, h)\sqsubseteq \widetilde{T}(g, h)$;
  \item[(O5)] $\widetilde{T}(\bm{1}_{[0, 1]}, \bm{1}_{[a, b]})=\bm{1}_{[0, b]}$;
  \item[(O6)] $\widetilde{T}$ is closed on $\mathbf{J}$;
  \item[(O7)] $\widetilde{T}$ is closed on $\mathbf{K}$.
\end{itemize}
A binary operation $\widetilde{S}: \mathbf{L}^2\rightarrow \mathbf{L}$ is a
{\it $t_r$-conorm} if it satisfies axioms (O1), (O2), (O4), (O6), and (O7) above;
axiom (O3'): $\widetilde{S}(f, \bm{1}_{\{0\}})=f$; and axiom (O5'):
$\widetilde{S}(\bm{1}_{[0, 1]}, \bm{1}_{[a, b]})=\bm{1}_{[a, 1]}$. Axioms (O1), (O2), (O3), (O3'),
and (O4) are called ``{\it basic axioms}'', and an operation that complies with
these axioms will be referred to as {\it $t$-norm} and {\it $t$-conorm}, respectively.
\end{definition}

\begin{definition}
For $f\in \mathbf{M}$, define
\begin{align*}
f^{L}(x)&=\sup\{f(y)\mid y\leq x\},\\
f^{L_{\mathrm{w}}}(x)&=\begin{cases}
\sup\{f(y)\mid y< x\}, & x\in (0, 1], \\
f(0), & x=0,
\end{cases}
\end{align*}
and
\begin{align*}
f^{R}(x)&=\sup\{f(y)\mid y\geq x\},\\
f^{R_{\mathrm{w}}}(x)&=\begin{cases}
\sup\{f(y)\mid y> x\}, & x\in [0, 1), \\
f(1), & x=1.
\end{cases}
\end{align*}
\end{definition}
Clearly, (1) $f^L$, $f^{L_{\mathrm{w}}}$ and $f^R$, $f^{R_{\mathrm{w}}}$ are {\color{blue}increasing} and
decreasing, respectively; (2) $f^{L}(x)\vee f^{R}(x)=f^{L}(x)\vee f^{R_{\mathrm{w}}}(x)=\sup_{z\in I}\{f(z)\}$ and
$f^{R}(x)\vee f^{L_{\mathrm{w}}}(x)=\sup_{z\in I}\{f(z)\}$ for all $x\in I$. The following properties of $f^{L}$
and $f^{R}$ {\color{blue}were obtained} by Walker et al. \cite{HWW2010,HWW2008,WW2005}.
\begin{proposition}{\rm \cite{WW2005}}\label{Main-Prop}
For $f, g\in \mathbf{M}$,
\begin{enumerate}[{\rm (1)}]
  \item $f\leq f^{L}\wedge f^{R}$;
  \item $(f^{L})^{L}=f^{L}$, $(f^{R})^{R}=f^{R}$;
  \item {\color{blue}$(\neg f)^{L}=\neg (f^{R})$, $(\neg f)^{R}=\neg (f^{L})$; }
  \item $(f^{L})^{R}=(f^{R})^{L}=\sup_{x\in I} \{f(x)\}$;
  \item $f\sqsubseteq g$ if and only if $f^{R}\wedge g \leq f\leq g^{R}$;
  \item $f\preccurlyeq g$ if and only if $f\wedge g^{L}\leq g\leq f^{L}$;
  \item $f$ is convex if and only if $f=f^{L}\wedge f^{R}$.
\end{enumerate}
\end{proposition}

\begin{theorem}{\rm \cite{HWW2010,HWW2008}}\label{order-theorem}
Let $f, g\in \mathbf{L}$. Then, $f\sqsubseteq g$ if and only if
$f^L\geq g^L$ and $f^R\leq g^R$.
\end{theorem}

{\color{blue}The following result follows from the definitions of $f^{L}$ and $f^{R}$.}

\begin{lemma}\label{L-R}
For $f\in \mathbf{M}$, $f^{L}(1)=f^{R}(0)=\sup_{x\in I}\{f(x)\}$.
\end{lemma}

\begin{proposition}\label{Pro-L}
For $f\in \mathbf{M}$, $f^{L_{\mathrm{w}}}(x)=\sup_{t\in [0, x)}\{f^{L}(t)\}$ for all $x\in (0, 1]$.
\end{proposition}

\begin{proof}
Fix any $x\in (0, 1]$, noting that $f(t)\leq f^{L}(t)$ for all $t\in [0, x)$, we have
$$
f^{L_{\mathrm{w}}}(x)=\sup_{t\in [0, x)}\{f(t)\}\leq\sup_{t\in [0, x)}\{f^{L}(t)\}.
$$

Moreover, for any $t\in [0, x)$, it follows from $t< \frac{t+x}{2}<x$ that
$f^{L}(t)\leq f^{L_{\mathrm{w}}}(\frac{t+x}{2})\leq f^{L_{\mathrm{w}}}(x)$,
implying that
$$
\sup_{t\in [0, x)}\{f^{L}(t)\}\leq f^{L_{\mathrm{w}}}(x).
$$
Thus,
$$
f^{L_{\mathrm{w}}}(x)=\sup_{t\in [0, x)}\{f^{L}(t)\}.
$$
\end{proof}

\begin{lemma}\label{<-Lemma}
For $f\in \mathbf{N}$, $\inf\{x\in I\mid f^{L}(x)=1\}\leq \sup\{x\in I\mid f^{R}(x)=1\}$.
\end{lemma}

\begin{proof}
From $f\in \mathbf{N}$ and Lemma~\ref{L-R}, it follows that $f^{L}(1)=f^{R}(0)=\sup\{f(x)\mid x\in I\}=1$,
which means that
both $\{x\in I\mid f^{L}(x)=1\}$ and $\{x\in I\mid f^{R}(x)=1\}$ are nonempty sets. Denote $\eta=
\inf\{x\in I\mid f^{L}(x)=1\}$ and $\xi=\sup\{x\in I\mid f^{R}(x)=1\}$. If $\eta=0$, this holds trivially.
If $\eta>0$, then for any $0\leq \alpha<\eta$, $f^{L}(\alpha)< 1$. This, together with
{\color{blue}$f^{L}(\alpha)\vee f^{R}(\alpha)=
\sup_{x\in I}\{f(x)\}=1$},
implies that $f^{R}(\alpha)=1$. Thus, $\alpha \leq \xi$. Therefore,
$$
\xi \geq \sup\{\alpha\mid 0\leq \alpha <\eta\}=\eta.
$$
\end{proof}


\section{\color{blue}Basic properties of $\ast$}\label{S3-1}
In this section, {\color{blue}basic properties} of $\ast$ determined by
the binary operations $\vartriangle$,
$\triangledown$, $\curlywedge$, and $\curlyvee$ are obtained.
\begin{proposition}\label{1-Lemma}
\begin{itemize}
\item[{\rm (1)}] Let $\ast$ be a $t$-norm on $I$. Then, $x\ast y=1$ if and only if $x=y=1$.
\item[{\rm (2)}] Let $\ast$ be a $t$-conorm on $I$. Then, $x\ast y=0$ if and only if $x=y=0$.
\end{itemize}
\end{proposition}

{\color{blue}\begin{lemma}\label{Con-1}
\begin{itemize}
\item[{\rm (1)}] Let $\vartriangle$ be a $t$-norm on $I$ and $\ast$ be a binary operation on $I$.
Then,
$$
(f\curlywedge g)(1)=f(1)\ast g(1).
$$
\item[{\rm (2)}] Let $\triangledown$ be a $t$-conorm on $I$ and $\ast$ be a binary operation on $I$.
Then,
$$
(f\curlyvee g)(0)=f(0)\ast g(0).
$$
\end{itemize}
\end{lemma}

\begin{proof}
Since $\vartriangle$ is a $t$-norm, from Proposition~\ref{1-Lemma}, we have
$$
(f\curlywedge g)(1)=\sup\{f(y)\ast g(z)\mid y\vartriangle z=1\}=f(1)\ast g(1).
$$
Similarly, we have
$$
(f\curlyvee g)(0)=\sup\{f(y)\ast g(z)\mid y\triangledown z=1\}=f(0)\ast g(0).
$$
\end{proof}
}

\begin{proposition}\label{Commu-Asso-Thm}
Let $\vartriangle$ be a $t$-norm on $I$ and $\ast$ be a binary operation on $I$.
If $\curlywedge$ is commutative on $\mathbf{L}$, then $\ast$ is commutative.
\end{proposition}
\begin{proof}
Suppose, on the contrary, that $\ast$ is not commutative. Then, there
exist $u, v\in I$ such that $u\ast v\neq v\ast u$.
Choose two functions $f, g\in \mathbf{M}$, as follows
$$
f(x)=(u-1)x+1,
$$
and
$$
g(x)=(v-1)x+1,
$$ for $x\in I$. It can be verified that
$f, g\in \mathbf{L}$, as both $f$ and $g$ are decreasing.
Since $\curlywedge$ is commutative, {\color{blue}Lemma~\ref{Con-1}} yields
\begin{align*}
&u\ast v=f(1) \ast g(1)=(f\curlywedge g)(1)\\
&=(g\curlywedge f)(1)=g(1) \ast f(1)=v\ast u,
\end{align*}
which is a contradiction. Therefore,
$\ast$ is commutative.
\end{proof}

\begin{proposition}\label{0-element}
Let $\vartriangle$ be a $t$-norm on $I$ and $\ast$ be a binary operation on $I$.
If $\curlywedge$ is a $t$-norm on $\mathbf{L}$, then
$0\ast 0=0\ast 1=1\ast 0=0$ and $1\ast 1=1$.
\end{proposition}

\begin{proof}
Since $\bm{1}_{\{1\}}$ is the neutral element of $\curlywedge$, from Lemma \ref{Con-1}
and Proposition~\ref{Commu-Asso-Thm},
it follows that
\begin{align*}
0&=\bm{1}_{\{0\}}(1)=(\bm{1}_{\{1\}}\curlywedge \bm{1}_{\{0\}})(1)\\
&=\bm{1}_{\{1\}}(1) \ast \bm{1}_{\{0\}}(1)\\
&=1\ast 0=0\ast 1;
\end{align*}
\begin{align*}
0&=\bm{1}_{\{0.5\}}(0)=(\bm{1}_{\{0.5\}}\curlywedge \bm{1}_{\{1\}})(0)\\
&\geq \bm{1}_{\{0.5\}}(1) \ast \bm{1}_{\{1\}}(0)\ (\text{as } 1\vartriangle 0=0)\\
& =0\ast 0;
\end{align*}
and
\begin{align*}
1&=\bm{1}_{\{1\}}(1)=(\bm{1}_{\{1\}}\curlywedge \bm{1}_{\{1\}})(1)\\
&=\bm{1}_{\{1\}}(1)\ast \bm{1}_{\{1\}}(1) =1\ast 1.
\end{align*}
\end{proof}

\begin{proposition}\label{0-element-2}
Let $\triangledown$ be a $t$-conorm on $I$ and $\ast$ be a binary operation on $I$.
If $\curlyvee$ is a $t$-conorm on $\mathbf{L}$, then
$0\ast 0=0\ast 1=1\ast 0=0$ and $1\ast 1=1$.
\end{proposition}
\begin{proof}
Since $\bm{1}_{\{0\}}$ is the neutral element of $\curlyvee$, from Lemma \ref{Con-1},
it follows that
\begin{equation}\label{eq-1-0-element-2}
\begin{split}
0&=\bm{1}_{\{1\}}(0)=(\bm{1}_{\{1\}}\curlyvee \bm{1}_{\{0\}})(0)\\
&=\bm{1}_{\{1\}}(0) \ast \bm{1}_{\{0\}}(0)=0\ast 1;
\end{split}
\end{equation}
\begin{equation}\label{eq-2-0-element-2}
\begin{split}
0&=\bm{1}_{\{1\}}(0)=(\bm{1}_{\{0\}}\curlyvee\bm{1}_{\{1\}})(0)\\
&=\bm{1}_{\{0\}}(0) \ast \bm{1}_{\{1\}}(0)=1\ast 0;
\end{split}
\end{equation}
and
\begin{equation}\label{eq-3-0-element-2}
\begin{split}
0&=\bm{1}_{\{0.5\}}(1)=(\bm{1}_{\{0.5\}}\curlyvee \bm{1}_{\{0\}})(1)\\
&\geq \bm{1}_{\{0.5\}}(0) \ast \bm{1}_{\{0\}}(1)\ (\text{as } 0\triangledown 1=1)\\
& =0\ast 0.
\end{split}
\end{equation}
It follows from \eqref{eq-1-0-element-2}--\eqref{eq-3-0-element-2} that,
for $y, z\in I$, one has $\bm{1}_{\{1\}}(y)\ast \bm{1}_{\{0\}}(z)\in \{0, 1\ast 1\}$.
This implies that
{\color{blue}\begin{align*}
1&=\bm{1}_{\{1\}}(1)=(\bm{1}_{\{1\}}\curlyvee \bm{1}_{\{0\}})(1)\\
&=\sup\{0, \bm{1}_{\{1\}}(1)\ast \bm{1}_{\{0\}}(0)\} \ (\text{as } 1\triangledown 0=1)\\
&=1\ast 1.
\end{align*}
}
\end{proof}


\begin{proposition}\label{product}
Let $\vartriangle$ be a $t$-norm on $I$ and $\ast$ be a binary operation on $I$.
If $\curlywedge$ is a $t$-norm on $\mathbf{L}$, then,
for $x_1, x_2\in I$, one has $\bm{1}_{\{x_1\}}\curlywedge \bm{1}_{\{x_2\}}=\bm{1}_{\{x_1\vartriangle x_2\}}$.
\end{proposition}

\begin{proof}
{\color{blue}Proposition~\ref{0-element} yields}
\begin{itemize}
  \item[(a)] for $y, z\in I$, $\bm{1}_{\{x_1\}}(y)\ast \bm{1}_{\{x_2\}}(z)\in\{0, 1\}$;
  \item[(b)] $\bm{1}_{\{x_1\}}(y)\ast \bm{1}_{\{x_2\}}(z)=1$ if and only if $y=x_1$ and $z=x_2$.
\end{itemize}
This, together with
$$(\bm{1}_{\{x_1\}}\curlywedge \bm{1}_{\{x_2\}})(x)
=\sup\{\bm{1}_{\{x_1\}}(y)\ast \bm{1}_{\{x_2\}}(z)\mid y\vartriangle z=x\},
$$
implies that
$$
\bm{1}_{\{x_1\}}\curlywedge \bm{1}_{\{x_2\}}=\bm{1}_{\{x_1\vartriangle x_2\}}.
$$
\end{proof}

\begin{proposition}\label{product-2}
Let $\triangledown$ be a $t$-conorm on $I$ and $\ast$ be a binary operation on $I$.
If $\curlyvee$ is a $t$-conorm, then, for $x_1, x_2\in I$, one has
$\bm{1}_{\{x_1\}}\curlyvee \bm{1}_{\{x_2\}}=\bm{1}_{\{x_1\triangledown x_2\}}$.
\end{proposition}

\begin{proof}
{\color{blue}Proposition~\ref{0-element-2} yields}
\begin{itemize}
  \item[(a)] for $y, z\in I$, $\bm{1}_{\{x_1\}}(y)\ast \bm{1}_{\{x_2\}}(z)\in\{0, 1\}$;
  \item[(b)] $\bm{1}_{\{x_1\}}(y)\ast \bm{1}_{\{x_2\}}(z)=1$ if and only if $y=x_1$ and $z=x_2$.
\end{itemize}
This, together with
$$(\bm{1}_{\{x_1\}}\curlyvee \bm{1}_{\{x_2\}})(x)
=\sup\{\bm{1}_{\{x_1\}}(y)\ast \bm{1}_{\{x_2\}}(z)\mid y\triangledown z=x\},
$$
implies that
$$
\bm{1}_{\{x_1\}}\curlyvee \bm{1}_{\{x_2\}}=\bm{1}_{\{x_1\triangledown x_2\}}.
$$
\end{proof}

\section{\color{blue}Construction of a $t_r$-norm on $\mathbf{L}$}\label{S-V}
{\color{blue}For any $f, g\in \mathbf{L}$, let $\eta_{f, g}=\inf\{x\in I\mid f^{L}(x)=1\}
\wedge \inf\{x\in I\mid g^{L}(x)=1\}$ and
$\xi_{f, g}=\sup\{x\in I\mid f^{R}(x)=1\}\wedge \sup\{x\in I\mid g^{R}(x)=1\}$. By Lemma~\ref{<-Lemma},
we have $\eta_{f, g} \leq \xi_{f, g}$.}
\begin{definition}\label{bingstar-operation}
Define a binary operation $\hollowstar: \mathbf{L}^2\rightarrow \mathbf{M}$ as follows: for $f, g\in \mathbf{L}$,
\begin{itemize}
  \item[(1)] $f=\bm{1}_{\{1\}}$, $f\hollowstar g=g\hollowstar f=g$;
  \item[(2)] $g=\bm{1}_{\{1\}}$, $f\hollowstar g=g\hollowstar f=f$;
  \item[(3)] $f\neq \bm{1}_{\{1\}}$ and $g\neq \bm{1}_{\{1\}}$,
\begin{equation}\label{xin-operation}
(f\hollowstar g)(t)=\begin{cases}
f^{L}(t)\vee g^{L}(t), & t\in [0, \eta_{f, g}),\\
1, & t\in [\eta_{f, g}, \xi_{f, g}),\\
f^{R}(\xi)\wedge g^{R}(\xi), & t=\xi_{f, g}, \\
0, & t\in (\xi_{f, g}, 1].
\end{cases}
\end{equation}
\end{itemize}
\end{definition}
Clearly, $f\hollowstar g$ is increasing on $[0, \xi_{f, g})$.

\begin{proposition}\label{Closed-Lemma}
For $f, g\in \mathbf{L}$,
$f\hollowstar g$ is normal and convex, i.e., $f\hollowstar g\in \mathbf{L}$.
\end{proposition}

\begin{proof}
Consider the following two cases:
\begin{enumerate}[(1)]
  \item if $f=\bm{1}_{\{1\}}$ or $g=\bm{1}_{\{1\}}$, it is clear that $f\hollowstar g\in \mathbf{L}$;
  \item if $f\neq \bm{1}_{\{1\}}$ and $g\neq \bm{1}_{\{1\}}$, applying (\ref{xin-operation}), it is easy to see
  that $f\hollowstar g$ is convex, {\color{blue}since it is increasing on $[0, \xi_{f, g})$ and decreasing on
  $[\xi_{f, g}, 1]$.} It remains to show that $f\hollowstar g$ is normal.
  \begin{enumerate}
    \item If $\eta_{f, g}< \xi_{f, g}$, then $(f\hollowstar g)(t)=1$ for all $t\in [\eta_{f, g}, \xi_{f, g})$;
    \item If $\eta_{f, g}= \xi_{f, g}$, consider the following two subcases:
    \begin{enumerate}
      \item[{\color{blue}(b.1)}] $\eta_{f, g}=0$. It follows from (\ref{xin-operation}) that
      $$
      (f\hollowstar g)(t)=
      \begin{cases}
      f^{R}(0)\wedge g^{R}(0), & t=0,\\
      0, & t\in (0, 1].
      \end{cases}
      $$
Since $f$ and $g$ are normal, from Lemma~\ref{L-R}, it is clear that
$$
f^{R}(0)\wedge g^{R}(0)=\sup_{x\in I}\{f(x)\}\wedge \sup_{x\in I}\{g(x)\}=1.
$$
      \item[{\color{blue}(b.2)}] $\eta_{f, g}>0$. From Proposition~\ref{Pro-L}, it follows that
      \begin{align*}
      &\quad (f\hollowstar g)^{L_{\mathrm{w}}}(\eta_{f, g})\\
      &=\sup_{t\in [0, \eta_{f, g})}\{(f\hollowstar g)(t)\}\\
      &=\sup_{t\in [0, \eta_{f, g})}\{f^{L}(t)\}\vee \sup_{t\in [0, \eta_{f, g})}\{g^{L}(t)\}\\
      &=f^{L_{\mathrm{w}}}(\eta_{f, g})\vee g^{L_{\mathrm{w}}}(\eta_{f, g}).
      \end{align*}
      If $f^{L_{\mathrm{w}}}(\eta_{f, g})\vee g^{L_{\mathrm{w}}}(\eta_{f, g})=1$, then clearly $f\hollowstar g$
      is normal. If $f^{L_{\mathrm{w}}}(\eta_{f, g})\vee g^{L_{\mathrm{w}}}(\eta_{f, g})<1$,
      noting that $1=\sup_{t\in I}\{f(t)\}
      =f^{L_{\mathrm{w}}}(\eta_{f, g})\vee f^{R}(\eta_{f, g})$ and $1=\sup_{t\in I}\{g(t)\}
      =g^{L_{\mathrm{w}}}(\eta_{f, g})\vee g^{R}(\eta_{f, g})$, we have
      $$
      f^{R}(\eta_{f, g})=g^{R}(\eta_{f, g})=1,
      $$
      which, together with $\eta_{f, g}=\xi_{f, g}$, implies that
      $$
      (f\hollowstar g)(\eta_{f, g})=1.
      $$
      \end{enumerate}
  \end{enumerate}
  Thus, $f\hollowstar g$ is normal.
\end{enumerate}
\end{proof}

\begin{remark}\label{R-24}
\begin{enumerate}[(i)]
  \item Proposition~\ref{Closed-Lemma} shows that the binary operation
$\hollowstar$ is closed on $\mathbf{L}$, i.e.,
$\hollowstar(\mathbf{L}^2) \subset \mathbf{L}$.
  \item From the proof of Proposition~\ref{Closed-Lemma}, it follows that, for $f, g\in \mathbf{L}$,
if $\eta_{f, g}=\xi_{f, g}$, then $(f\hollowstar g)^{L_{\mathrm{w}}}(\xi_{f, g})=1$
or $(f\hollowstar g)(\xi_{f, g})=1$.
\end{enumerate}
\end{remark}

\begin{proposition}\label{LR*-Thm}
For $f, g\in \mathbf{L}\backslash \{\bm{1}_{\{1\}}\}$,
\begin{align}
(f\hollowstar g)^{L}(t)&=
\begin{cases}
f^{L}(t)\vee g^{L}(t), & t\in [0, \eta_{f, g}),\\
1, & t\in [\eta_{f, g}, 1],
\end{cases}\label{eq-LP}\\
(f\hollowstar g)^{R}(t)&=
\begin{cases}
1, & t\in [0, \xi_{f, g}),\\
f^{R}(\xi_{f, g})\wedge g^{R}(\xi_{f, g}), & t=\xi_{f, g}, \\
0, & t\in (\xi_{f, g}, 1].
\end{cases}\label{eq-LQ}
\end{align}
\end{proposition}

\begin{proof}
(1) If $\eta_{f, g}<\xi_{f, g}$, since $f^{L}(t)\vee g^{L}(t)$ is increasing,
{\color{blue}(\ref{xin-operation}) evidently implies \eqref{eq-LP} and \eqref{eq-LQ}.}

\medskip

(2) If $\eta_{f, g}=\xi_{f, g}$, {\color{blue}the result follows from Remark \ref{R-24}
(ii) and (\ref{xin-operation}).}
\end{proof}

\begin{theorem}\label{hollowstar-O1}
$\hollowstar$ satisfies (O1).
\end{theorem}

\begin{proof}
For $f, g\in \mathbf{L}$,

\medskip

{\color{blue}(A.1)} if $f=\bm{1}_{\{1\}}$ or $g=\bm{1}_{\{1\}}$, then clearly $f\hollowstar g=g\hollowstar f$;

\medskip

{\color{blue}(A.2)} if $f\neq \bm{1}_{\{1\}}$ and $g\neq \bm{1}_{\{1\}}$, then
$$
(f\hollowstar g)(t)=\begin{cases}
f^{L}(t)\vee g^{L}(t), & t\in [0, \eta_{f, g}),\\
1, & t\in [\eta_{f, g}, \xi_{f, g}),\\
f^{R}(\xi_{f, g})\wedge g^{R}(\xi_{f, g}), & t=\xi_{f, g},\\
0, & t\in (\xi_{f, g}, 1],
\end{cases}
$$
{\color{blue}and the commutativity of $\hollowstar$
follows from the commutativity of $\vee$ and $\wedge$.}
\end{proof}

{\color{blue}\begin{lemma}\label{no=1}
For $f, g\in \mathbf{L}\backslash \{\bm{1}_{\{1\}}\}$, $f\hollowstar g \neq \bm{1}_{\{1\}}$.
\end{lemma}

\begin{proof}
Suppose on the contrary that $f\hollowstar g=\bm{1}_{\{1\}}$. Then, $\eta_{f, g}=\xi_{f, g}=1$ and
$f^{L}(t)\vee g^{L}(t)=0$ for $t\in [0, 1)$. Since $f^{L}\geq f$ and $g^{L}\geq g$,
we have $f(t)=g(t)=0$ for $t\in [0, 1)$. This, together with $f, g\in \mathbf{L}$, implies that
$$
f=g=\bm{1}_{\{1\}},
$$
which is a contradiction.
\end{proof}
}

\begin{theorem}\label{S-4B}
$\hollowstar$ satisfies (O2).
\end{theorem}

\begin{proof}
For $f, g, h\in \mathbf{L}$,

\medskip

{\color{blue}(B.1)} if one of $f$, $g$, and $h$ is equal to $\bm{1}_{\{1\}}$,
then it is easy to verify that $(f\hollowstar g) \hollowstar h
  =f\hollowstar (g\hollowstar h)$;

\medskip

{\color{blue}(B.2)} if none of $f$, $g$, and $h$ are equal to $\bm{1}_{\{1\}}$, then
$$
(f\hollowstar g)(t)=\begin{cases}
f^{L}(t)\vee g^{L}(t), & t\in [0, \eta_{f, g}),\\
1, & t\in [\eta_{f, g}, \xi_{f, g}),\\
f^{R}(\xi_{f, g})\wedge g^{R}(\xi_{f, g}), & t=\xi_{f, g},\\
0, & t\in (\xi_{f, g}, 1],
\end{cases}
$$
and
$$
(g\hollowstar h)(t)=\begin{cases}
g^{L}(t)\vee h^{L}(t), & t\in [0, \eta_{g, h}),\\
1, & t\in [\eta_{g, h}, \xi_{g, h}),\\
g^{R}(\xi_{g, h})\wedge h^{R}(\xi_{g, h}), & t=\xi_{g, h},\\
0, & t\in (\xi_{g, h}, 1].
\end{cases}
$$

{\color{blue}By Lemma~\ref{no=1}, we have $f\hollowstar g\neq \bm{1}_{\{1\}}$
and $g\hollowstar h\neq \bm{1}_{\{1\}}$.}

{\color{blue}Proposition~\ref{LR*-Thm} implies that}
\begin{align*}
(f\hollowstar g)^{L}(t)&=
\begin{cases}
f^{L}(t)\vee g^{L}(t), & t\in [0, \eta_{f, g}),\\
1, & t\in [\eta_{f, g}, 1],
\end{cases}\\
(f\hollowstar g)^{R}(t)&=
\begin{cases}
1, & t\in [0, \xi_{f, g}),\\
f^{R}(\xi_{f, g})\wedge g^{R}(\xi_{f, g}), & t=\xi_{f, g}, \\
0, & t\in (\xi_{f, g}, 1],
\end{cases}
\end{align*}
and
\begin{align*}
(g\hollowstar h)^{L}(t)&=
\begin{cases}
g^{L}(t)\vee h^{L}(t), & t\in [0, \eta_{g, h}),\\
1, & t\in [\eta_{g, h}, 1],
\end{cases}\\
(g\hollowstar h)^{R}(t)&=
\begin{cases}
1, & t\in [0, \xi_{g, h}),\\
g^{R}(\xi_{g, h})\wedge h^{R}(\xi_{g, h}), & t=\xi_{g, h}, \\
0, & t\in (\xi_{g, h}, 1].
\end{cases}
\end{align*}
Since $f\hollowstar g, g\hollowstar h, f, h\in \mathbf{L}\backslash\{\bm{1}_{\{1\}}\}$, we have
\begin{equation}\label{A-1}
\begin{split}
&\quad((f\hollowstar g)\hollowstar h)(t)\\
&=\begin{cases}
(f\hollowstar g)^{L}(t)\vee h^{L}(t), & t\in [0, \eta_{f\hollowstar g, h}),\\
1, & t\in [\eta_{f\hollowstar g, h}, \xi_{f\hollowstar g, h}),\\
(f\hollowstar g)^{R}(\xi_{f\hollowstar g, h})\wedge h^{R}(\xi_{f\hollowstar g, h}), & t=\xi_{f\hollowstar g, h},\\
0, & t\in (\xi_{f\hollowstar g, h}, 1],
\end{cases}
\end{split}
\end{equation}
and
\begin{equation}\label{A-2}
\begin{split}
&\quad(f\hollowstar (g\hollowstar h))(t)\\
&=\begin{cases}
f^{L}(t)\vee (g\hollowstar h)^{L}(t), & t\in [0, \eta_{f, g\hollowstar h}),\\
1, & t\in [\eta_{f, g\hollowstar h}, \xi_{f, g\hollowstar h}),\\
f^{R}(\xi_{f, g\hollowstar h})\wedge (g\hollowstar h)^{R}(\xi_{f, g\hollowstar h}), & t=\xi_{f, g\hollowstar h}, \\
0, & t\in (\xi_{f, g\hollowstar h}, 1].
\end{cases}
\end{split}
\end{equation}
Clearly,
\begin{align*}
\eta&:=\eta_{f\hollowstar g, h}=\eta_{f, g\hollowstar h}\\
&\ =\inf\{x\in I\mid f^{L}(x)=1\}\\
&\quad\wedge \inf\{x\in I\mid g^{L}(x)=1\}\\
&\quad\wedge \inf\{x\in I\mid h^{L}(x)=1\}\\
&\ = \eta_{f, g}\wedge \eta_{g, h},
\end{align*}
and
\begin{align*}
\xi&:=\xi_{f\hollowstar g, h}=\xi_{f, g\hollowstar h}\\
&\ =\sup\{x\in I\mid f^{R}(x)=1\}\\
& \quad\wedge \sup\{x\in I\mid g^{R}(x)=1\}\\
& \quad\wedge \sup\{x\in I\mid h^{R}(x)=1\}\\
&\ = \xi_{f, g}\wedge \xi_{g, h}.
\end{align*}
Thus, for $t\in [0, \eta)$,
$$
((f\hollowstar g)\hollowstar h)(t)=(f^{L}(t)\vee g^{L}(t)) \vee h^{L}(t),
$$
and
$$
(f\hollowstar (g\hollowstar h))(t)=f^{L}(t)\vee (g^{L}(t) \vee h^{L}(t)),
$$
{\color{blue}and the associativity holds.}
Clearly, for $t\in [\eta, \xi)\cup (\xi, 1]$,
$$
((f\hollowstar g)\hollowstar h)(t)=(f\hollowstar (g\hollowstar h))(t).
$$

\medskip

{\color{blue}To finish the proof we have to show that
$((f\hollowstar g)\hollowstar h)(\xi)=(f\hollowstar (g\hollowstar h))(\xi)$.}

\medskip

Consider the following three cases:

\medskip

{\color{blue}(B.2.1)} If $\xi_{f, g}=\xi_{g, h}$, then $\xi=\xi_{f, g}=\xi_{g, h}$, implying that
$$
(f\hollowstar g)^{R}(\xi)=f^{R}(\xi)\wedge g^{R}(\xi),
$$
and
$$
(g\hollowstar h)^{R}(\xi)=g^{R}(\xi)\wedge h^{R}(\xi).
$$
{\color{blue}Then, (\ref{A-1}) and (\ref{A-2}) yield}
$$
((f\hollowstar g)\hollowstar h)(\xi)=f^{R}(\xi)\wedge g^{R}(\xi)\wedge h^{R}(\xi)
= (f\hollowstar (g\hollowstar h))(\xi).
$$

{\color{blue}(B.2.2)} If $\xi_{f, g}<\xi_{g, h}$, then $\xi_{f, g}=\sup\{x\in I\mid f^{R}(x)=1\}<\sup\{x\in I\mid g^{R}(x)=1\}$
(as $\sup\{x\in I\mid f^{R}(x)=1\}\geq\sup\{x\in I\mid g^{R}(x)=1\}$ implies that $\xi_{f, g}=\sup\{x\in I\mid g^{R}(x)=1\}
\geq \sup\{x\in I\mid g^{R}(x)=1\}\wedge \sup\{x\in I\mid h^{R}(x)=1\}=\xi_{g, h}$), which means
that there exists $\hat{x}>\xi_{f, g}$ such that $g^{R}(\hat{x})=1$.
Thus,
$$
g^{R}(\xi_{f, g})\geq g^{R}(\hat{x})=1.
$$
{\color{blue}Therefore, since $\xi=\xi_{f, g}\wedge \xi_{g, h}=\xi_{f, g}$, we get}
\begin{equation}\label{A-3}
(f\hollowstar g)^{R}(\xi)=(f\hollowstar g)^{R}(\xi_{f, g})=f^{R}(\xi_{f, g})\wedge g^{R}(\xi_{f, g})=f^{R}(\xi_{f, g}).
\end{equation}
From $\xi_{f, g}<\xi_{g, h}\leq \sup\{x\in I\mid h^{R}(x)=1\}$, it follows that there exists
$x'>\xi_{f, g}$ such that $h^{R}(x')=1$, implying that
\begin{equation}\label{QST-1}
h^{R}(\xi)=h^{R}(\xi_{f, g})\geq h^{R}(x')=1.
\end{equation}
\eqref{QST-1} together with (\ref{A-1}) and (\ref{A-3}) implies that
$$
((f\hollowstar g)\hollowstar h)(\xi)=f^{R}(\xi).
$$

Since $\xi=\xi_{f, g}<\xi_{g, h}$, then we have $(g\hollowstar h)^{R}(\xi)=1$, which together with (\ref{A-2})
implies that
$$
(f\hollowstar (g\hollowstar h))(\xi)=f^{R}(\xi)\wedge (g\hollowstar h)^{R}(\xi)=f^{R}(\xi).
$$
Therefore,
$$
((f\hollowstar g)\hollowstar h)(\xi)=(f\hollowstar (g\hollowstar h))(\xi)=f^{R}(\xi).
$$

{\color{blue}(B.2.3)} {\color{blue}If $\xi_{f, g}>\xi_{g, h}$,
then similarly as in the previous case we can show that
$$
((f\hollowstar g)\hollowstar h)(\xi)=(f\hollowstar (g\hollowstar h))(\xi)=h^{R}(\xi).
$$
}

Summing up (B.2.1)--(B.2.3), we have
$$
((f\hollowstar g)\hollowstar h)(\xi)=(f\hollowstar (g\hollowstar h))(\xi).
$$
\end{proof}

\begin{theorem}\label{hollowstar-O3}
$\hollowstar$ satisfies (O3).
\end{theorem}

\begin{proof}
This follows directly from (1) and (2) of Definition~\ref{bingstar-operation}.
\end{proof}

\begin{theorem}
$\hollowstar$ satisfies (O4).
\end{theorem}

\begin{proof}
{\color{blue}We have to show that, for $f, g, h\in \mathbf{L}$ with $f\sqsubseteq g$,
$f\hollowstar h\sqsubseteq g\hollowstar h$. There are the following possible cases:}

\medskip

{\color{blue}(D.1)} if $h=\bm{1}_{\{1\}}$, then $f\hollowstar h=f\sqsubseteq g =g\hollowstar h$;

\medskip

{\color{blue}(D.2)} if $f=\bm{1}_{\{1\}}$, then $g=\bm{1}_{\{1\}}$ (as $f\sqsubseteq g$),
implying that $f\hollowstar h=h\sqsubseteq h =g\hollowstar h$;

\medskip

{\color{blue}(D.3)} if $g=\bm{1}_{\{1\}}$ and $f\neq \bm{1}_{\{1\}} \neq h$, then
  $$
(f\hollowstar h)(t)=
\begin{cases}
f^{L}(t)\vee h^{L}(t), & t\in [0, \eta_{f, h}),\\
1, & t\in [\eta_{f, h}, \xi_{f, h}),\\
f^{R}(\xi_{f, h})\wedge h^{R}(\xi_{f, h}), & t=\xi_{f, h},\\
0, & t\in (\xi_{f, h}, 1].
\end{cases}
$$
{\color{blue}By
Proposition~\ref{LR*-Thm},} one has
\begin{equation}
(f\hollowstar h)^{L}(t)=
\begin{cases}
f^{L}(t)\vee h^{L}(t), & t\in [0, \eta_{f, h}),\\
1, & t\in [\eta_{f, h}, 1],
\end{cases}
\end{equation}
and
\begin{equation}\label{O-4}
(f\hollowstar h)^{R}(t)=
\begin{cases}
1, & t\in [0, \xi_{f, h}),\\
f^{R}(\xi_{f, h})\wedge h^{R}(\xi_{f, h}), & t=\xi_{f, h},\\
0, & t\in (\xi_{f, h}, 1].
\end{cases}
\end{equation}
Clearly,
$$
(f\hollowstar h)^{L}\geq h^{L},
$$
and
\begin{equation}\label{e-*}
(f\hollowstar h)^{R}(\xi_{f, h}) \leq h^{R}(\xi_{f, h}).
\end{equation}
{\color{blue}Additionally,} for $t<\xi_{f, h}$, it follows from
$\xi_{f, h} \leq \sup\{x\in I\mid h^{R}(x)=1\}$ that there exists
$t<\hat{x}\leq \xi$ such that $h^{R}(\hat{x})=1$. Thus,
$h^{R}(t)\geq h^{R}(\hat{x})=1$ since $h^{R}$ is decreasing,
implying that, for $t\in [0, \xi)$,
{\color{blue}\begin{equation}\label{ST-1}
(f\hollowstar h)^{R}(t) \leq 1=h^{R}(t).
\end{equation}
\eqref{ST-1} together with (\ref{e-*}) and (\ref{O-4}) implies that}
$$
(f\hollowstar h)^{R}\leq h^{R}.
$$
{\color{blue}By Theorem~\ref{order-theorem} there is}
$$
f\hollowstar h\sqsubseteq  h=g\hollowstar h.
$$

{\color{blue}(D.4)} If $f\neq \bm{1}_{\{1\}}$, $g\neq \bm{1}_{\{1\}}$,
and $h\neq \bm{1}_{\{1\}}$, then
from the definition of $\hollowstar$, it follows that
$$
(f\hollowstar h)(t)=
\begin{cases}
f^{L}(t)\vee h^{L}(t), & t\in [0, \eta_{f, h}),\\
1, & t\in [\eta_{f, h}, \eta_{f, h}),\\
f^{R}(\eta_{f, h})\wedge h^{R}(\eta_{f, h}), & t=\eta_{f, h}, \\
0, & t\in (\eta_{f, h}, 1],
\end{cases}
$$
and
 $$
(g\hollowstar h)(t)=
\begin{cases}
g^{L}(t)\vee h^{L}(t), & t\in [0, \eta_{g, h}),\\
1, & t\in [\eta_{g, h}, \eta_{g, h}),\\
g^{R}(\eta_{g, h})\wedge h^{R}(\eta_{g, h}), & t=\eta_{g, h}, \\
0, & t\in (\eta_{g, h}, 1].
\end{cases}
$$

From $f\sqsubseteq g$ and Theorem \ref{order-theorem}, it follows that $f^{L}\geq g^{L}$
and $f^{R}\leq g^{R}$. Therefore,
$$
\{x\in I\mid g^{L}(x)=1\}\subseteq \{x\in I\mid f^{L}(x)=1\},
$$
and
$$
\{x\in I\mid f^{R}(x)=1\}\subseteq \{x\in I\mid g^{R}(x)=1\},
$$
implying that
$$
\inf\{x\in I\mid g^{L}(x)=1\}\geq \inf\{x\in I\mid f^{L}(x)=1\},
$$
and
$$
\sup\{x\in I\mid f^{R}(x)=1\}\leq \sup\{x\in I\mid g^{R}(x)=1\}.
$$
Thus,
$$
\eta_{f, h}\leq \eta_{g, h} \text{ and } \eta_{f, h}\leq \eta_{g, h}.
$$

{\color{blue}Further, by Proposition~\ref{LR*-Thm} there is}
\begin{align*}
(f\hollowstar h)^{L}(t)&=
\begin{cases}
f^{L}(t)\vee h^{L}(t), & t\in [0, \eta_{f, h}),\\
1, & t\in [\eta_{f, h}, 1],
\end{cases}\\
(f\hollowstar h)^{R}(t)&=
\begin{cases}
1, & t\in [0, \eta_{f, h}),\\
f^{R}(\eta_{f, h})\wedge h^{R}(\eta_{f, h}), & t=\eta_{f, h}, \\
0, & t\in (\eta_{f, h}, 1],
\end{cases}
\end{align*}
and
\begin{align*}
(g\hollowstar h)^{L}(t)&=
\begin{cases}
g^{L}(t)\vee h^{L}(t), & t\in [0, \eta_{g, h}),\\
1, & t\in [\eta_{g, h}, 1],
\end{cases}\\
(g\hollowstar h)^{R}(t)&=
\begin{cases}
1, & t\in [0, \eta_{g, h}),\\
g^{R}(\eta_{g, h})\wedge h^{R}(\eta_{g, h}), & t=\eta_{g, h}, \\
0, & t\in (\eta_{g, h}, 1].
\end{cases}
\end{align*}
From $f^{R}\leq g^{R}$, it follows that
\begin{equation}\label{O-5}
(f\hollowstar h)^{R}\leq (g\hollowstar h)^{R}.
\end{equation}
From $f^{L}\geq g^{L}$, it follows that, for $t\in [0, \eta_{f, h})$,
$$
(f\hollowstar h)^{L}(t)\geq (g\hollowstar h)^{L}(t).
$$
It is clear that, for $t\in [\eta_{f, h}, 1]$,
$$
(f\hollowstar h)^{L}(t)=1 \geq (g\hollowstar h)^{L}(t).
$$
Thus,
\begin{equation}\label{QST-2}
(f\hollowstar h)^{L}\geq (g\hollowstar h)^{L}.
\end{equation}
\eqref{QST-2} together with (\ref{O-5}) and Theorem~\ref{order-theorem} implies that
$$
f\hollowstar h\sqsubseteq g\hollowstar h.
$$
\end{proof}

\begin{theorem}
$\hollowstar$ satisfies (O5).
\end{theorem}

\begin{proof}
{\color{blue}Assume $0\leq a\leq b\leq 1$. Then we can distinguish the following cases:}

\medskip

{\color{blue}(E.1)} If $a=1$, then $\bm{1}_{[0, 1]}\hollowstar \bm{1}_{[a, b]}
=\bm{1}_{[0, 1]} \hollowstar \bm{1}_{\{1\}}=\bm{1}_{[0, 1]}$;

\medskip

{\color{blue}(E.2)} If $a<1$, then
  $$
  \bm{1}_{[0, 1]}^{L}\equiv 1,
  $$
  $$
  \bm{1}_{[0, 1]}^{R}\equiv 1,
  $$
  $$
  \bm{1}_{[a, b]}^{L}(x)=\begin{cases}
  0, & x\in [0, a),\\
  1, & x\in [a, 1],
  \end{cases}
  $$
  and
  $$
  \bm{1}_{[a, b]}^{R}(x)=\begin{cases}
  1, & x\in [0, b],\\
  0, & x\in (b, 1],
  \end{cases}
  $$
which implies that
  $\inf\{x\in I\mid \bm{1}_{[0,1]}^{L}(x)=1\}\wedge \inf\{x\in I\mid \bm{1}_{[a, b]}^{L}(x)=1\}=0$ and
  $\sup\{x\in I\mid \bm{1}_{[0, 1]}^{R}(x)=1\}\wedge \sup\{x\in I\mid \bm{1}_{[a, b]}^{R}(x)=1\}=b$.

\medskip

  Consider the following two subcases:

\medskip

{\color{blue}(E.2.1)} If $b=0$, we have
  $$(\bm{1}_{[0, 1]}\hollowstar \bm{1}_{[a, b]})(t)=\begin{cases}
\bm{1}_{[0, 1]}^{L}(t)\vee \bm{1}_{[a, b]}^{L}(t), & t\in [0, 0),\\
1, & t=0,\\
0, & t\in (0, 1],
\end{cases}
$$
implying that $\bm{1}_{[0, 1]}\hollowstar \bm{1}_{[a, b]}=\bm{1}_{[0, 0]}=\bm{1}_{[0, b]}$
as $ [0, 0)=\emptyset$.

\medskip

{\color{blue}(E.2.2)} If $b>0$, we have
\begin{equation}\label{E.2.2}
(\bm{1}_{[0, 1]}\hollowstar \bm{1}_{[a, b]})(t)
=\begin{cases}
\bm{1}_{[0, 1]}^{L}(t)\vee \bm{1}_{[a, b]}^{L}(t), & t\in [0, 0),\\
1, & t\in [0, b),\\
\bm{1}_{[0, 1]}^{R}(b)\wedge \bm{1}_{[a, b]}^{R}(b), & t=b,\\
0, & t\in (b, 1].
\end{cases}
\end{equation}
\eqref{E.2.2} together with $\bm{1}_{[0, 1]}^{R}(b)\wedge \bm{1}_{[a, b]}^{R}(b)=1$
implies that $\bm{1}_{[0, 1]}\hollowstar \bm{1}_{[a, b]}=\bm{1}_{[0, b]}$.
\end{proof}

\begin{theorem}\label{Theorem-min}
{\color{blue}For $x_1, x_2\in I$, $\bm{1}_{\{x_1\}}\hollowstar \bm{1}_{\{x_2\}}=\bm{1}_{\{x_1\wedge x_2\}}$.
In particular, $\hollowstar$ satisfies (O6).}
\end{theorem}

\begin{proof}
{\color{blue}Clearly, $\bm{1}_{\{x_1\}}\hollowstar \bm{1}_{\{x_2\}}=\bm{1}_{\{x_1\}}$ when $x_2=1$ by
Theorem~\ref{hollowstar-O3}.}

Moreover, for $x\in I$,
$$
\bm{1}_{\{x\}}^{L}(t)=
\begin{cases}
0, & t\in [0, x), \\
1, & t\in [x, 1],
\end{cases}
$$
and
$$
\bm{1}_{\{x\}}^{R}(t)=
\begin{cases}
1, & t\in [0, x], \\
0, & t\in (x, 1].
\end{cases}
$$
Then, for $0\leq x_1\leq  x_2< 1$, $\inf\{x\in I\mid \bm{1}_{\{x_1\}}^{L}(x)=1\}\wedge
\inf\{x\in I\mid \bm{1}_{\{x_2\}}^{L}(x)=1\}=x_1$ and
$\sup\{x\in I\mid \bm{1}_{\{x_1\}}^{R}(x)=1\}\wedge \sup\{x\in I\mid \bm{1}_{\{x_2\}}^{R}(x)=1\}=x_1$. Clearly,
$(\bm{1}_{\{x_1\}}\hollowstar \bm{1}_{\{x_2\}})^{L_{\mathrm{w}}}(x_1)=0$. From Remark~\ref{R-24} (ii), it follows
that
$$
(\bm{1}_{\{x_1\}}\hollowstar \bm{1}_{\{x_2\}})(t)=
\begin{cases}
0, & t\in [0, x_1), \\
1, & t= x_1, \\
0, & t\in (x_1, 1],
\end{cases}
$$
which, together with the commutativity of $\hollowstar$, implies that
$$
\bm{1}_{\{x_1\}}\hollowstar \bm{1}_{\{x_2\}}=\bm{1}_{\{x_2\}}\hollowstar
\bm{1}_{\{x_1\}}=\bm{1}_{\{x_1\}}\in \mathbf{J}.
$$
\end{proof}

\begin{theorem}\label{hollowstar-O7}
{\color{blue}For $[a_1, b_1]$, $[a_2, b_2]\subset I$, $\bm{1}_{[a_1, b_1]}\hollowstar \bm{1}_{[a_2, b_2]}
=\bm{1}_{[a_1\wedge a_2, b_1\wedge b_2]}$.
In particular, $\hollowstar$ satisfies (O7).}
\end{theorem}

\begin{proof}
{\color{blue}Clearly, $\bm{1}_{[a_1, b_1]}\hollowstar \bm{1}_{[a_2, b_2]}
=\bm{1}_{[a_2, b_2]}\hollowstar \bm{1}_{[a_1, b_1]}\in \mathbf{K}$
when $[a_1, b_1]=\{1\}$ or $[a_2, b_2]=\{1\}$ by Theorem~\ref{hollowstar-O3}.

Moreover, for $0\leq a\leq b\leq 1$,}
$$
  \bm{1}_{[a, b]}^{L}(t)=
  \begin{cases}
  0, & t\in [0, a),\\
  1, & t\in [a, 1],
  \end{cases}
  $$
  and
  $$
  \bm{1}_{[a, b]}^{R}(t)=
  \begin{cases}
  1, & t\in [0, b],\\
  0, & t\in (b, 1].
  \end{cases}
  $$
Then, for $[a_1, b_1]$, $[a_2, b_2]\subset I$ with $[a_1, b_1]\neq \{1\}$ and $[a_2, b_2]\neq \{1\}$,
we have $\inf\{x\in I\mid \bm{1}_{[a_1, b_1]}^{L}(x)=1\}\wedge \inf\{x\in I\mid \bm{1}_{[a_2, b_2]}^{L}(x)=1\}
=a_1\wedge a_2$ and $\sup\{x\in I\mid \bm{1}_{[a_1, b_1]}^{R}(x)=1\}\wedge
\sup\{x\in I\mid \bm{1}_{[a_2, b_2]}^{R}(x)=1\}=b_1\wedge b_2$.
From (\ref{xin-operation}), it follows that
\begin{align*}
&\quad(\bm{1}_{[a_1, b_1]}\hollowstar \bm{1}_{[a_2, b_2]})(t)\\
&=
\begin{cases}
\bm{1}_{[a_1, b_1]}^{L}(t)\vee \bm{1}_{[a_2, b_2]}^{L}(t), & t\in [0, a_1\wedge a_2), \\
1, & t\in [a_1\wedge a_2, b_1\wedge b_2), \\
\bm{1}_{[a_1, b_1]}^{R}(t)\wedge \bm{1}_{[a_2, b_2]}^{R}(t),
& t=b_1\wedge b_2,\\
0, & t\in (b_1\wedge b_2, 1],
\end{cases}\\
&=\begin{cases}
0, & t\in [0, a_1\wedge a_2), \\
1, & t\in [a_1\wedge a_2, b_1\wedge b_2], \\
0, & t\in (b_1\wedge b_2, 1],
\end{cases}
\end{align*}
which, together with the commutativity of $\hollowstar$, implies that
$$
\bm{1}_{[a_1, b_1]}\hollowstar \bm{1}_{[a_2, b_2]}=\bm{1}_{[a_2, b_2]}\hollowstar \bm{1}_{[a_1, b_1]}=
\bm{1}_{[a_1\wedge a_2, b_1\wedge b_2]}\in \mathbf{K}.
$$
\end{proof}


{\color{blue}Theorems~\ref{hollowstar-O1}--\ref{hollowstar-O7} imply the following result.}

\begin{theorem}\label{tr-norm-theorem}
The binary operation $\hollowstar$ is a $t_{r}$-norm on $\mathbf{L}$.
\end{theorem}

\section{$\hollowstar$ cannot be obtained by $\curlywedge$ and $\curlyvee$}

This section shows that the $t_{r}$-norm $\hollowstar$ constructed in Section~\ref{S-V}
cannot be obtained by operations $\curlywedge$ or $\curlyvee$.

The following theorem provides a
sufficient condition ensuring that $\ast$ is a $t$-norm on $I$.

\begin{theorem}\label{WC-TFS}{\rm \cite[Theorem~21]{WC-TFS}}
Let $\vartriangle$ be a continuous $t$-norm on $I$ and $\ast$ be a binary operation on $I$.
If $\curlywedge$ is a $t$-norm on $\mathbf{L}$, then $\ast$ is a $t$-norm.
\end{theorem}

\begin{theorem}\label{Thm-30}
For any binary operation $\ast$ on $I$ and any $t$-norm $\vartriangle$ on $I$,
there exist $f, g\in \mathbf{L}$ such that $f\hollowstar g \neq f\curlywedge g$,
i.e., $\hollowstar$ cannot be obtained by $\curlywedge$.
\end{theorem}

\begin{proof}
Suppose, on the contrary, that there exist a binary operation $\ast$
on $I$ and a $t$-norm $\vartriangle$ on $I$
such that, for any $f, g\in \mathbf{L}$, one has $f\hollowstar g=f\curlywedge g$.

\medskip

{\color{blue}First we will show that
$\vartriangle =\wedge$.}

\medskip

For $x_1, x_2\in I$, {\color{blue}Theorem~\ref{Theorem-min} gives}
$$
\bm{1}_{\{x_1\}}\hollowstar \bm{1}_{\{x_2\}}=\bm{1}_{\{x_1\wedge x_2\}}.
$$
{\color{blue}Further, Theorem~\ref{tr-norm-theorem} and Proposition~\ref{product} yield}
$$
\bm{1}_{\{x_1\}}\hollowstar \bm{1}_{\{x_2\}}=\bm{1}_{\{x_1\}}\curlywedge
\bm{1}_{\{x_2\}}=\bm{1}_{\{x_1\vartriangle x_2\}}.
$$
Thus,
{\color{blue}$$
x_1 \wedge x_2= x_1\vartriangle x_2 \text{ for all } x_1,x_2\in I, \text{ i.e., }
\vartriangle =\wedge.
$$
}

Clearly, $\vartriangle=\wedge$ is a continuous $t$-norm on $I$.
From Theorem~\ref{tr-norm-theorem} and Theorem~\ref{WC-TFS}, it follows
that $\ast$ is a $t$-norm on $I$ and
\begin{equation}\label{O-7}
(f\hollowstar g)(x)=\sup\{f(y)\ast g(z)\mid y\wedge z=x\}.
\end{equation}
Choose a function $\psi\in \mathbf{L}$ by
$$
\psi(x)=
\begin{cases}
1, & x\in [0, 0.75],\\
0.5, & x\in (0.75, 1].
\end{cases}
$$
\begin{figure}[H]
\begin{center}
\scalebox{0.6}{\includegraphics{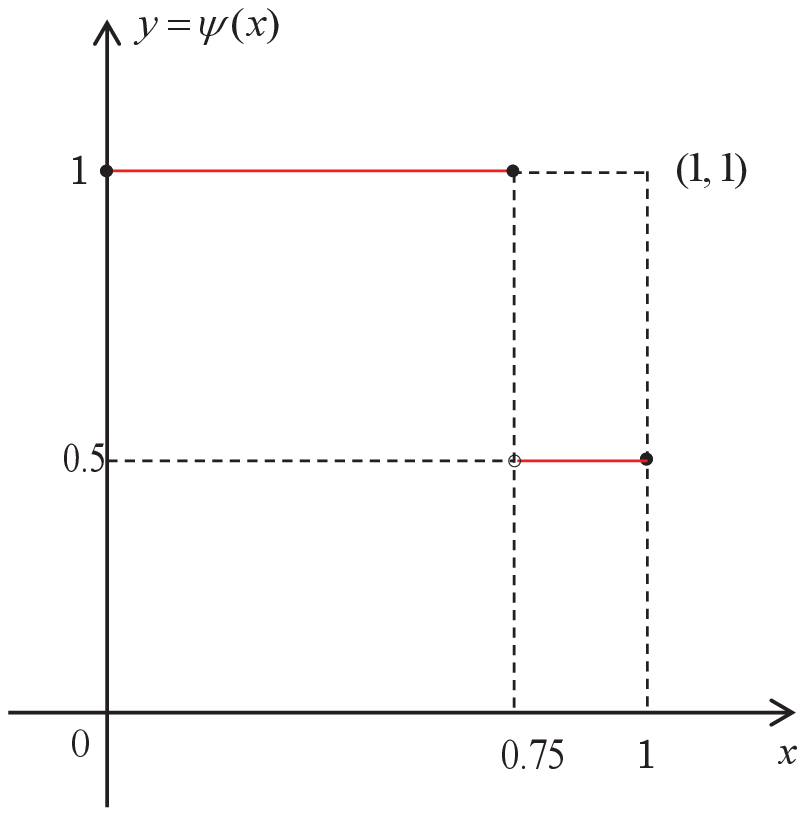}}
\renewcommand{\figure}{Fig.}
\caption{\color{blue}The function $\psi$.
}
\end{center}
\end{figure}
{\color{blue}Then,}
$$
\psi^{L}(x)\equiv 1,
$$
\begin{align*}
\psi^{R}(x)&=
\begin{cases}
1, & x\in [0, 0.75],\\
0.5, & x\in (0.75, 1],
\end{cases}\\
\bm{1}_{\{0.8\}}^{L}(x)&=
\begin{cases}
0, & x\in [0, 0.8), \\
1, & x\in [0.8, 1],
\end{cases}
\end{align*}
and
$$
\bm{1}_{\{0.8\}}^{R}(x)=
\begin{cases}
1, & x\in [0, 0.8], \\
0, & x\in (0.8, 1].
\end{cases}
$$
From~(\ref{xin-operation}), we have
$$
(\psi\hollowstar \bm{1}_{\{0.8\}})(t)=
\begin{cases}
1, & x\in [0, 0.75],\\
0, & x\in (0.75, 1].
\end{cases}
$$
In particular,
\begin{equation}\label{QST-3}
(\psi\hollowstar \bm{1}_{\{0.8\}})(0.8)=0.
\end{equation}
\eqref{QST-3} together with (\ref{O-7}) and the fact that $\ast$ is a $t$-norm implies that
\begin{align*}
0&=(\psi\hollowstar \bm{1}_{\{0.8\}})(0.8)\\
&=\sup\left\{\psi(y)\ast \bm{1}_{\{0.8\}}(z)\mid y\wedge z=0.8\right\}\\
&\geq
\psi(0.8)\ast \bm{1}_{\{0.8\}}(0.8)
=0.5\ast 1=0.5,
\end{align*}
which is a contradiction.
\end{proof}

{\color{blue}\begin{theorem}\label{not-norm}
Let $\ast$ be a binary operation on $I$, $\vartriangle$ be a
$t$-norm on $I$, and $\triangledown$ be a $t$-conorm on $I$.
Then, we have
\begin{itemize}
  \item[{\rm (1)}] the binary operation $\curlyvee$ defined by \eqref{O-2} is
  not a $t$-norm on $\mathbf{L}$;
  \item[{\rm (2)}] the binary operation $\curlywedge$ defined by \eqref{O-1} is
  not a $t$-conorm on $\mathbf{L}$.
\end{itemize}
\end{theorem}

\begin{proof}
(1) Suppose on the contrary that $\curlyvee$ is a $t$-norm on $\mathbf{L}$. Then, we have
$f\curlyvee \bm{1}_{\{1\}}=f$ for any $f\in \mathbf{L}$ as $\bm{1}_{\{1\}}$
is the neutral element of $\curlyvee$. For $0\leq \zeta\leq 1$, take $f_{\zeta}: I\to I$
as $f_{\zeta}(x)=(1-\zeta)x+\zeta$. Clearly, $f_{\zeta}\in \mathbf{L}$, which together with
Lemma~\ref{Con-1} implies that
\begin{equation}\label{QST-4}
\zeta=f_{\zeta}(0)=(f_{\zeta}\curlyvee \bm{1}_{\{1\}})(0)=f_{\zeta}(0)\ast \bm{1}_{\{1\}}(0)
=\zeta \ast 0.
\end{equation}
Similarly, we have $0\ast \zeta=\zeta$. Moreover, for $g\in \mathbf{L}$ with $g(x)=1-x$ ($x\in I$),
we have
\begin{align*}
0.5&=g(0.5)=(g\curlyvee \bm{1}_{\{1\}})(0.5)\\
& =\sup\{g(y)\ast \bm{1}_{\{1\}}(z)\mid y\triangledown z=0.5\}\\
& \geq g(0) \ast \bm{1}_{\{1\}}(0.5)\ (\text{as } 0\triangledown 0.5=0.5)\\
& =g(0)\ast 0=1 \ (\text{by } \eqref{QST-4}),
\end{align*}
which is a contradiction.

\medskip

(2) Similarly as in the previous case we can prove that $\curlywedge$ is
  not a $t$-conorm on $\mathbf{L}$.
\end{proof}

Theorems~\ref{tr-norm-theorem} and \ref{not-norm} imply the following result.

\begin{corollary}\label{Thm-31}
The $t_{r}$-norm $\hollowstar$ cannot be obtained by $\curlyvee$.
\end{corollary}
}

\begin{remark}
{\color{blue}Theorems~\ref{tr-norm-theorem} and \ref{Thm-30}, and Corollary~\ref{Thm-31} show} that
there exists a $t_{r}$-norm $\hollowstar$ on $\mathbf{L}$, which cannot be obtained using the
formulas defining the operations $\curlywedge$ and $\curlyvee$. This gives a positive answer
to Question~\ref{Q-1}.
\end{remark}

\section{A $t_r$-conorm that is not obtained by $\curlywedge$ and $\curlyvee$}\label{S-VI}
{\color{blue}This section introduces the dual operation for every binary operation on $\mathbf{M}$ and proves
the duality between $t_{r}$-norm and $t_{r}$-conorm, also between $\curlywedge$ and $\curlyvee$.
As an application, a $t_{r}$-conorm on $\mathbf{L}$, which cannot be obtained by $\curlywedge$
and $\curlyvee$, is obtained.}

\begin{definition}\label{Def-Com}
Let $\lozenge$ be a binary operation on $\mathbf{M}$. Define the {\it dual operation}
$\lozenge^{\complement}$ of $\lozenge$ as follows: for $f, g\in \mathbf{M}$,
$$
f \lozenge^{\complement}g=\neg ((\neg f) \lozenge (\neg g)).
$$
\end{definition}

\begin{proposition}\label{Dual-Prop}
For a binary operation $\lozenge$ on $\mathbf{M}$, $(\lozenge^{\complement})^{\complement}=\lozenge$.
\end{proposition}
\begin{proof}
For $f, g\in \mathbf{M}$, from Definition~\ref{Def-Com}, it follows that
$f(\lozenge^{\complement})^{\complement} g=\neg((\neg f)\lozenge^{\complement} (\neg g))=
\neg (\neg (f\lozenge g))=f \lozenge g$.
\end{proof}

\begin{theorem}\label{Dual-Thm}
Let $\lozenge$ be a binary operation on $\mathbf{M}$ such that $\lozenge(\mathbf{L}^2)\subset \mathbf{L}$,
{\color{blue}i.e., $\lozenge$ is closed on $\mathbf{L}$.}
Then, $\lozenge|_{\mathbf{L}^2}$ is a $t_{r}$-norm (resp., $t$-norm) on $\mathbf{L}$ if and only if
$\lozenge^{\complement}|_{\mathbf{L}^2}$ is a $t_{r}$-conorm (resp., $t$-conorm) on $\mathbf{L}$.
\end{theorem}

\begin{proof}
Clearly,
$\lozenge^{\complement}$ is closed on $\mathbf{L}$. By Proposition~\ref{Dual-Prop}, it suffices to
show that $\lozenge^{\complement}|_{\mathbf{L}^2}$
is a $t_{r}$-conorm provided that $\lozenge|_{\mathbf{L}^2}$ is a $t_{r}$-norm.
\begin{enumerate}[(i)]
\item $\lozenge^{\complement}$ satisfies (O1).

For $f, g\in \mathbf{L}$, since $\lozenge$ satisfies (O1), we have
$f\lozenge^{\complement} g=\neg ((\neg f)\lozenge (\neg g))=\neg ((\neg g)\lozenge (\neg f))
=g\lozenge^{\complement} f$.

\item $\lozenge^{\complement}$ satisfies (O2).

For $f, g, h\in \mathbf{L}$, {\color{blue}we get}
$$(f\lozenge^{\complement} g) \lozenge^{\complement} h
=\neg ((\neg(f \lozenge^{\complement} g)) \lozenge (\neg h))
=\neg ( ((\neg f) \lozenge (\neg g)) \lozenge (\neg h)),
$$
and
$$
f\lozenge^{\complement} (g \lozenge^{\complement} h)
=\neg((\neg f)\lozenge (\neg (g\lozenge^{\complement} h)))=
\neg ((\neg f) \lozenge ((\neg g)\lozenge (\neg h))),
$$
{\color{blue}and the associativity of $\lozenge^{\complement}$
follows from the associativity of $\lozenge$.}

\item $\lozenge^{\complement}$ satisfies (O3').

For $f\in \mathbf{L}$, since $\bm{1}_{\{1\}}$ is the neutral element of $\lozenge$,
we have
$$
f\lozenge^{\complement} \bm{1}_{\{0\}}=\neg ((\neg f) \lozenge (\neg \bm{1}_{\{0\}}))
=\neg ((\neg f)\lozenge \bm{1}_{\{1\}})=\neg (\neg f)=f.
$$

\item $\lozenge^{\complement}$ satisfies (O4).

{\color{blue}For $f, g, h\in \mathbf{L}$ with $f\sqsubseteq g$, Proposition~\ref{Main-Prop} and
Theorem~\ref{order-theorem} yield
$$
(\neg f)^{L}(x)=(\neg (f^{R}))(x)
=f^{R}(1-x)\leq g^{R}(1-x)=(\neg (g^{R}))(x)=(\neg g)^{L}(x),
$$
and
$$
(\neg f)^{R}(x)=(\neg (f^{L}))(x)
=f^{L}(1-x)\geq g^{L}(1-x)=(\neg (g^{L}))(x)=(\neg g)^{R}(x).
$$
Applying again Theorem~\ref{order-theorem}, we obtain $\neg g\sqsubseteq \neg f$.}
Since $\lozenge$ satisfies (O4), we have
$$
(\neg g)\lozenge (\neg h)\sqsubseteq (\neg f)\lozenge (\neg h).
$$
Thus,
$$
f\lozenge^{\complement} h=\neg ((\neg f)\lozenge (\neg h))
\sqsubseteq \neg ((\neg g)\lozenge (\neg h))=g\lozenge^{\complement} h.
$$

\item $\lozenge^{\complement}$ satisfies (O5').

Since $\lozenge$ satisfies (O5), it follows that
\begin{align*}
&\quad\bm{1}_{[0, 1]}\lozenge^{\complement} \bm{1}_{[a, b]}\\
&=\neg ((\neg \bm{1}_{[0, 1]})\lozenge (\neg \bm{1}_{[a, b]}))\\
&=\neg (\bm{1}_{[0, 1]}\lozenge \bm{1}_{[1-b, 1-a]})\\
&=\neg \bm{1}_{[0, 1-a]}=\bm{1}_{[a, 1]}.
\end{align*}

\item $\lozenge^{\complement}$ satisfies (O6).

{\color{blue}For $x_1, x_2\in I$, since $\lozenge$ satisfies (O6), then there exists $x_{3}\in I$
such that $\bm{1}_{\{1-x_1\}}\lozenge \bm{1}_{\{1-x_2\}}=\bm{1}_{\{x_3\}}$, implying that
\begin{align*}
&\quad\bm{1}_{\{x_1\}}\lozenge^{\complement} \bm{1}_{\{x_2\}}\\
&=\neg ((\neg \bm{1}_{\{x_1\}})\lozenge (\neg \bm{1}_{\{x_2\}}))\\
&=\neg (\bm{1}_{\{1-x_1\}}\lozenge \bm{1}_{\{1-x_2\}})\\
&=\neg \bm{1}_{\{x_3\}}=\bm{1}_{\{1-x_3\}}\in \mathbf{J}.
\end{align*}
}

\item $\lozenge^{\complement}$ satisfies (O7).

{\color{blue}For $[a_1, b_1]$, $[a_2, b_2]\subset I$, since $\lozenge$ satisfies (O7), then
there exist $[a_3, b_3]\subset I$ such that $\bm{1}_{[1-b_1, 1-a_1]}\lozenge \bm{1}_{[1-b_2, 1-a_2]}
=\bm{1}_{[a_3, b_3]}$, implying that
\begin{align*}
&\quad\bm{1}_{[a_1, b_1]}\lozenge^{\complement} \bm{1}_{[a_2, b_2]}\\
&=\neg ((\neg \bm{1}_{[a_1, b_1]})\lozenge (\neg \bm{1}_{[a_2, b_2]}))\\
&=\neg (\bm{1}_{[1-b_1, 1-a_1]}\lozenge \bm{1}_{[1-b_2, 1-a_2]})\\
&=\neg \bm{1}_{[a_3, b_3]}=\bm{1}_{[1-b_3, 1-a_3]}\in \mathbf{K}.
\end{align*}
}
\end{enumerate}
\end{proof}

{\color{blue}\begin{theorem}\label{Dual-obtain}
A binary operation on $\mathbf{M}$ is obtained by $\curlywedge$ if and only if its dual operation
is obtained by $\curlyvee$.
\end{theorem}

\begin{proof}
By Proposition~\ref{Dual-Prop}, it suffices to prove the necessity. Assume that $\curlywedge$
is a binary operation satisfying that there exist a binary operation $\ast$ on $I$ and a
$t$-norm $\vartriangle$ on $I$ such that, for $f, g\in \mathbf{M}$,
$(f\curlywedge g)(x)=\sup\{f(y)\ast g(z)\mid y\vartriangle z=x\}$. Take $\triangledown: I^{2}\to I$
as $x \triangledown y=1-(1-x)$$\vartriangle$$(1-y)$ for any $(x, y)\in I^{2}$. Clearly, $\triangledown$
is a $t$-conorm on $I$. For any $f, g\in \mathbf{M}$ and $x\in I$, we have
\begin{align*}
(f\curlywedge^{\complement}g)(x)&=(\neg ((\neg f)\curlywedge (\neg g)))(x)\\
&=((\neg f)\curlywedge (\neg g))(1-x)\\
&=\sup\{(\neg f)(y)\ast (\neg g)(z)\mid y\vartriangle z=1-x\}\\
&=\sup\{f(1-y)\ast g(1-z)\mid y\vartriangle z=1-x\}\\
&=\sup\{f(y)\ast g(z)\mid y\triangledown z=x\},
\end{align*}
implying that $\curlywedge^{\complement}$ can be obtained by $\curlyvee$.
\end{proof}

Theorems~\ref{tr-norm-theorem}, \ref{Dual-Thm}, and \ref{Dual-obtain} imply the following result.

\begin{corollary}\label{Tr-conorm-Thm}
Let $\fivestar=\hollowstar^{\complement}$. Then,
\begin{itemize}
  \item[{\rm (1)}] $\fivestar$ is
a $t_{r}$-conorm on $\mathbf{L}$;
  \item[{\rm (2)}] $\fivestar$ cannot be obtained by $\curlywedge$ and $\curlyvee$.
\end{itemize}
\end{corollary}
}

\section{Conclusion}
Employing the functions $f^{L}$ and $f^{R}$, we have constructed in this paper two binary operations
$\hollowstar$ and $\fivestar=\hollowstar^{\complement}$ on $\mathbf{L}$ (see Definitions
\ref{bingstar-operation} and \ref{Def-Com}) and proved that $\hollowstar$ is a $t_r$-norm
on $\mathbf{L}$ and $\fivestar$
is a $t_{r}$-conorm on $\mathbf{L}$. Both $\hollowstar$ and $\fivestar$ cannot be obtained by
using the formulas defining the operations $\curlywedge$ and $\curlyvee$. These results give
a positive answer to an open problem (see Question~\ref{Q-1}) in \cite{HCT2015}. Combining
this result with our main results in \cite{WC-TFS}, the two open problems posed by
Hern\'{a}ndez et al. \cite{HCT2015} are completely solved.

\section*{Acknowledgements}

The authors express their sincere thanks to the editors and reviewers for
their most valuable comments and suggestions in improving this paper greatly.

\end{document}